\documentclass{amsart}
\usepackage[leqno]{amsmath}

\usepackage{amsmath}
\usepackage{amsthm}
\usepackage{amssymb}
\usepackage{amscd}
\usepackage{amsxtra}     % Use various AMS packages
\usepackage{epsfig}
\usepackage{verbatim}
\usepackage[all]{xypic}

\theoremstyle{plain} %default

\newtheorem{thm}{Theorem}[section]
\newtheorem{lem}[thm]{Lemma}
\newtheorem{prop}[thm]{Proposition}

\newtheorem{cor}[thm]{Corollary}
\newtheorem{dfn}[thm]{Definition}

\theoremstyle{definition}
\newtheorem{eg}[thm]{Example}

\newtheorem{ques}[thm]{Question}
\newtheorem{rmk}[thm]{Remark}

\numberwithin{equation}{section}

\newcommand{\ZZ}{\mathbb{Z}}

\newcommand{\tensor}{\otimes}

 \DeclareMathOperator{\ann}{ann}
 
 \DeclareMathOperator{\Tor}{Tor}
 \DeclareMathOperator{\Ext}{Ext}
 \DeclareMathOperator{\Hom}{Hom}

 \DeclareMathOperator{\Supp}{Supp}
 \DeclareMathOperator{\Spec}{Spec}
 \DeclareMathOperator{\Sing}{Sing}
 \DeclareMathOperator{\Se}{S}

 \DeclareMathOperator{\pd}{pd}
 
 \DeclareMathOperator{\cx}{cx}
 \DeclareMathOperator{\lcx}{lcx}
 \DeclareMathOperator{\tcx}{tcx}
 
 \DeclareMathOperator{\depth}{depth}

 \DeclareMathOperator{\Z}{Z}
 \DeclareMathOperator{\h}{\eta}

\newcommand{\al}{a\ell}

\begin{document}

\bibliographystyle{plain}

\title{Asymptotic behavior of Tor over complete intersections and applications}
%\author{Hai Long Dao}
\author{Hailong Dao}

\address{Department of Mathematics, University of Utah,  155 South 1400 East, Salt Lake City,
UT 84112-0090, USA} \email{hdao@math.utah.edu}

\maketitle

\begin{abstract}
Let $R$ be a local complete intersection and $M,N$ are $R$-modules
such that $\ell(\Tor_i^R(M,N))<\infty$ for $i\gg 0$. Imitating an
approach by Avramov and Buchweitz, we investigate the asymptotic
behavior of $\ell(\Tor_i^R(M,N))$ using Eisenbud operators and show
that they have well-behaved growth. We define and study a function
$\eta^R(M,N)$ which generalizes  Serre's intersection multiplicity
$\chi^R(M,N)$ over regular local rings and Hochster's function
$\theta^R(M,N)$ over local hypersurfaces. We use good properties of
$\eta^R(M,N)$ to obtain various results on complexities of $\Tor$
and $\Ext$, vanishing of $\Tor$, depth of tensor products, and
dimensions of intersecting modules over local complete
intersections.
\end{abstract}

\section{Introduction}
In this article we define and study a certain function on pairs of
modules $(M,N)$ over a local complete intersection $R$. We first
recall some notions which inspire our work. In 1961, Serre defined a
notion of intersection multiplicity for two finitely generated
modules $M,N$  over a regular local ring $R$ with
$\ell(M\tensor_RN)<\infty$ as:
$$ \chi^R(M,N) = \displaystyle \sum_{i\geq0} (-1)^i\ell(\Tor_i^R(M,N)) $$

In 1980, Hochster defined a function $\theta^R(M,N)$ for a pair of
finitely generated modules $M,N$ over a local hypersurface $R$ such
that $\ell(\Tor_i^R(M,N))<\infty$ for $i\gg0$ as:
$$ \theta^R(M,N) = \ell(\Tor_{2e+2}^R(M,N)) - \ell(\Tor_{2e+1}^R(M,N)) .$$
where $e$ is an integer such that $e\geq d/2$. It is well known (see
\cite{Ei}) that $\Tor^R(M,N)$ is periodic of period at most 2 after
$d+1$ spots, so this function is well-defined. The vanishing of this
function over certain hypersurfaces was shown by Hochster to imply
the Direct Summand Conjecture. This function is related to what is
called ``Herbrand difference" by Buchweitz in \cite{Bu}. Hochster's
theta function has recently been exploited in \cite{Da1,Da2,Da4} to
study a number of different questions on hypersurfaces, giving new
results on rigidity of $\Tor$, dimensions of intersecting cycles,
depth of $\Hom$ and tensor products, splitting of vector bundles.
These results extend works in \cite{AG2,Au1,Au2,Fa,HW1,HW2,Jot,MNP}
and provide some surprising new links between the classical
homological questions that have been an active part of Commutative
Algebra in the last 50 years.

Our main goal is to define, for a local complete intersection $R$, a
function on any pair of $R$-modules $(M,N)$ satisfying
$\ell(\Tor_i^R(M,N))<\infty$ for $i\gg0 $  which can be viewed as a
generalization of  both Serre's intersection multiplicity and
Hochster's theta function. Our definition will be asymptotic, since
over complete intersections, free resolutions of modules do not have
any obvious ``finite" or ``periodic" property. Therefore, a study of
the growth of lengths of the $\Tor$ modules is an essential first
step.

Our approach for such task is parallel to that of a recent beautiful
paper by Avramov and Buchweitz (\cite{AB1}). The theory of
complexity, which measures the polynomial growth of the Betti
numbers of a module, has long been an active subject of Commutative
Algebra. In their paper, Avramov and Buchweitz  studied complexity
for $\Ext$ modules using Quillen's approach to cohomology of finite
groups and the structure of the total module $\Ext^*(M,N)$ as a
Noetherian module over the ring of cohomology operators (in the
sense of Eisenbud, see \cite{Ei}). One of the technical difficulties
for our approach  is that over the ring of cohomology operators, it
is not clear what structure  $\Tor_*(M,N)$ has. It is well known
that if $\ell(M\tensor_RN)<\infty$ then $\Tor_*(M,N)$ becomes an
Artinian module, however our condition on lengths of $\Tor$ modules
is weaker.

We begin, in the second section, by introducing a notion that is
suitable to describing the structure of $\Tor_*(M,N)$. Let $T =
\bigoplus_i T_i$ be a $\mathbb{N}$-graded module over
$S=R[x_1,..,x_r]$ such that the $x_i$s act with equal negative
degree. $T$ is called {\it{almost Artinian}} if there is an integer
$j$ such that $T_{\geq j} = \bigoplus_{i\geq j}T_i$  is Artinian
over $S$. We collect basic properties for modules in this category.
We also introduce a number of different notions of complexities for
pairs of modules using $\Tor$.

Section \ref{almost_art} is devoted to an investigation of the
structure of the module $\Tor_*(M,N)$ over the ring of homology
operators. When $R = Q/(f_1,...,f_r)$ (here $Q$ is not necessarily
regular), these operators are $R$-linear maps $ x_j:
\Tor_{i+2}^R(M,N) \to \Tor_i^R(M,N) $ for $1 \leq j \leq r$ and
$i\geq 0$. We show that when $\ell(\Tor_i^R(M,N)) <\infty$ for
$i\gg0$, the module $\Tor_*^R(M,N)$ is almost Artinian over the ring
of homology operators $S = R[x_1,...,x_r]$. We also show that  when
$\ell(\Tor_i^R(M,N)) <\infty$ for $i\gg0$, $\Tor_*^R(M,N)$ is almost
Artinian over $S$ if and only if $\Tor_*^Q(M,N)$ is almost Artinian
over $Q$, that is, $\Tor_i^Q(M,N)=0$ when $i\gg0$. This is an
analogue of results obtained by Gulliksen, Avramov-Gasharov-Peeva
and is crucial in our analysis of lengths of $\Tor$ modules later.
Our proof is rather {\it ad hoc}, since almost Artinian modules do
not behave as well as Artinian or Noetherian ones, (see Example
\ref{exalmostart}).

In section \ref{eta} we study the ``adjusted lengths" (which is
equal to normal length, except when the module does not have finite
length, then it is equal to $0$) of $\Tor_i^R(M,N)$, which we call
{\it{generalized Betti numbers}} $\beta_i(M,N)$. We prove properties
for these numbers which subsume previously known results about Betti
numbers of a module over local complete intersections. We arrive at
the main goal of this note, the definition of a function:
$$\h_e^R(M,N):= \lim_{n\to\infty} \frac{\sum_0^{n} (-1)^i \beta_i(M,N)}{n^e}$$
We show that $\h_e^R(M,N)$ is finite when $e$ is at least the
``Tor-complexity" of $(M,N)$, and it is additive on short exact
sequences provided that it is defined on all pairs involved. We also
obtain a change of rings result which relates the values of
$\h^R(M,N)$ and $\h^{R/(f)}(M,N)$ where $f$ is a nonzerodivisor on
$R$.

In  section \ref{comparison}, we compare the several notions of
complexities that arise in our work with the one previously studied
by Avramov and Buchweitz. We prove they coincide when all the higher
$\Tor$ modules have finite length, generalizing (see \ref{AB2}) a
striking result in \cite{AB1} that over a complete intersection, the
vanishing of all higher $\Tor$ modules is equivalent to the
vanishing of all higher $\Ext$ modules. The connections between the
complexities in general seems like a difficult problem and are worth
further investigation.

The rest of the paper is concerned with applications. In section
\ref{c-rigid} we study vanishing behavior of  $\Tor_i^R(M,N)$. The
key idea is to use our results about rigidity over hypersurface in
\cite{Da1} as the base case and our change of rings theorem for
$\eta^R(M,N)$ for the inductive step. Our results in this situation
improve on results in  \cite{Mu,Jo1,Jo2}. In section \ref{appsec},
the main results basically say that under some extra conditions,
good depth of  $M\tensor_RN$ forces the vanishing of $\Tor_i^R(M,N)$
for all $i>0$. These are generalizations of Auslander's classical
result on tensor product over regular local rings and in some senses
improve upon similar results in complete intersections of small
codimensions by Huneke, Jorgensen, and Wiegan in \cite{HJW}.
Finally, in section \ref{intersection} we give some applications on
intersection theory over local complete intersections, extending
results by Hochster and Roberts in \cite{Ho1,Ro}. We also discuss
some interesting questions that give new perspectives on a classical
conjecture that grew out of Serre's work on intersection
multiplicity. We remark that the above list  contains only the most
obvious applications of $\h^R(M,N)$. More technical ones, for
example generalizations of results by Auslander and
Auslander-Goldman, will be the topics of forthcoming papers.

The main ideas in this article can almost certainly be applied to
study vanishing of $\Ext$ modules over complete intersections. In
fact, the module structure of $\Ext^*_R(M,N)$ over the ring of
cohomology operators is much better understood (see \cite{AB1, AB2,
AGP}). We focus our study on $\Tor$, since many of the open
homological questions could be viewed as problems about length of
$\Tor$ modules.

This article grew out of parts of the author's PhD thesis at the
University of Michigan. The author would like to thank his advisor,
Melvin Hochster, for constant support and encouragement. Some of the
work here was done when the author was visiting the Mathematics
Department at University of Nebraska, Lincoln in May 2005. Special
thanks must go to Luchezar Avramov and Srikanth Ieyengar for some
very helpful discussions.

\section{Notation and preliminary results}\label{prelim}

Throughout this section, let $(R,m,k)$ be an Noetherian local ring.
Let $M,N$ be finitely generated $R$-modules. Let $S=R[x_1,..,x_r]$
for indeterminates $x_1,...,x_r$ and $T =  \bigoplus_{i\geq 0} T_i$
be a $\mathbb{N}$-graded module over $S$. We shall start by making
or recalling some definitions.

We define the \textit{adjusted length} of $M$ as :
%\begin{displaymath}
$$          \al_R(M) :=\left\{ \begin{array}{ll}
          \ell_R(M) & \text{if $\dim(M) =0$}\\
           0   & \text{if $\dim(M) >0$}
          \end{array} \right.
$$
The minimal number of generators of $M$ is denoted by $\mu(M)$.

We define the \textit{finite length index} of $T$ as :
$$f_R(T):=\inf \{ i |\ \ell_R(T_j)<\infty \ \text{for $j \geq i$} \} $$
%$$ f_R((M,N) := \inf \{ i |\ \ell(\Tor_j^T(M,N))<\infty \ \text{for $j \geq i$} \} $$
(If no such $i$ exists, we set $f_R(T) = \infty$).

The \textit{complexity} of a sequence of integers $B = \{b_i\}_{i\geq 0}$ is defined as:\\
 $\cx(B) := \inf \ \{d\in Z \mid b_n\leq an^{d-1} \ \text
{for some real number} \  a  \ \text {and all} \ n\gg 0 \}$

If $f_R(T)<\infty$ then one can define the complexity of $T$ as:
$$\cx_R(T) := \cx(\{al_R(T_i)\})$$

For a pair of $R$-modules $M,N$, let $\Tor_*^R(M,N)
=\bigoplus_i\Tor_i^R(M,N)$ and
$\Ext^*_R(M,N)=\bigoplus_i\Ext^i_R(M,N)$. The concept of
\textit{complexity} for $(M,N)$ was first introduced in \cite{AB1}.
In our notations, their definition becomes:
$$\cx_R(M,N) := \cx_R(\Ext^*_R(M,N)\tensor_R k) $$

Similarly we will define several analogues of $\cx_R(M,N)$. The
\textit{Tor complexity} of $M,N$ is:
$$\tcx_R(M,N) = \cx_R(\Tor_*^R(M,N)\tensor_R k) $$
The \textit{length complexity} of $M,N$ is:
$$\lcx_R(M,N) = \cx_R(\Tor_*^R(M,N))$$

\begin{dfn}\label{almostdef}
Let $T =  \bigoplus_{i\geq 0} T_i$ be a $\mathbb{N}$-graded module
over $S=R[x_1,..,x_r]$ such that the $x_i$s act with equal negative
degree. $T$ is called almost Artinian over $S$ if there is an
integer $j$ such that $T_{\geq j} = \bigoplus_{i\geq j}T_i$  is
Artinian over $S$.
\end{dfn}

Now we collect some  results that will be needed. Some were already
in  the literature, but for some we can not find a reference. First,
we recall the long exact sequence for change of rings
(\ref{longexact}). It follows from the  Cartan-Eilenberg spectral
sequence (\cite{Av2}, 3.3.2).

\begin{prop}\label{longexact}
Let $R=Q/(f)$ such that $f$ is a nonzerodivisor on $Q$, and let
$M,N$ be $R$-modules. Then we have the long exact sequence of
$\Tor$s :

$$        \begin{array}{ll}
...\to \Tor_{n}^R(M,N) \to \Tor_{n+1}^Q(M,N) \to \Tor_{n+1}^R(M,N)\\
\to  \Tor_{n-1}^R(M,N) \to \Tor_{n}^Q(M,N) \to \Tor_{n}^R(M,N)\\
\to ...                                                         \\
\to \Tor_{0}^R(M,N) \to \Tor_{1}^Q(M,N) \to \Tor_{1}^R(M,N) \to 0

\end{array}
$$\\
\end{prop}

\begin{thm}\label{kirby}(\cite{Ki}, theorem 1)
Suppose that all the $x_i$ have equal negative degree. Then $T$ is
an Artinian $S$ module if and only if there are integers $n,l$ such
that :
\\(1) $T_i=0$ for $i<l$
\\(2) $0:_{T_i} (\sum_{j=1}^r x_jR) = 0$ for $i>n$.
\\(3) $T_i$ is an Artinian $R$-module for all $i$. \\
Now suppose that the $x_i$ have equal positive degree $d$. Then $T$ is an Noetherian $S$ module
if and only in there are integer $n,l$ such that:
\\(4) $T_i=0$ for $i<l$
\\(5) $T_{i+d} = \sum_{j=1}^r x_jT_i$ for all $i>n$.
\\(6) $T_i$ is an Noetherian $R$-module for all $i$.
\end{thm}

\begin{cor}\label{cor1}
Assume that the $x_i$ have equal positive degree $d$ and $T$ is
Noetherian over $S$. Then the sequence of ideals $\{\ann_R(T_i)\}$
eventually becomes periodic of period $d$.
\end{cor}

\begin{proof}
By \ref{kirby} you can choose $n$ such that $T_{i+d} = \sum_{j=1}^r
x_jT_i$ for all $i>n$. Let $J_i = \ann_R(T_i)$. Then $J_{i}
\subseteq J_{i+d}$ for $i>n$. Since $R$ satisfies ACC, the
conclusion follows.
\end{proof}

\begin{cor}\label{3.4}
Let $T,T'$ be such that $\ell(T_i),\ell(T'_i)<\infty$ for all $i$
and $T_i = T'_{i+e}$ for some $e$ and all $i\gg 0$. Then $T$ is
almost Artinian if and only if $T$ is  Artinian if and only if $T'$
is Artinian.
%$T$ is strongly k-Artinian over $S$ if and only if  $T'$ is.
\end{cor}

\begin{lem}\label{3.9}
Let $0\to T' \to T \to T'' \to 0$ be a short exact sequence of $S$
modules such as the maps are homogenous. Then $T$ is almost Artinian
if and only if $T',T''$ are almost Artinian.
%\\ii)If any two of the modules are strongly k-Artinian, so is the third one.
\end{lem}

\begin{proof}
This is obvious from the definitions.
%For ii), apply $\tensor_R k$ to the short exact sequence.
\end{proof}

\begin{lem}\label{eisen}(\cite{Ei}, 3.3)
Let $R=k$ be a field and assume $k$ is infinite. Also assume that all the $x_i$ are
of equal negative degree $-d$, and $T$ is an Artinian $S$ module. Then there are elements
$\alpha_i\in k$ such that multiplication by $x=x_r + \sum_{i=2}^{r-1}\alpha_i x_i$ induces
a surjective map $T_{i+d} \to T_i$ for all $i\gg 0$.
\end{lem}

\begin{lem}\label{gulik}(\cite{Gu}, 1.3)
Let $x: T\to T$ be a homogenous $S$-linear map of negative degree. Assume that $\ker(x)$ is
an Artinian $S$-module. Then $T$ is an Artinian $S[x]$-module.
\end{lem}

\begin{eg}\label{exalmostart}
Gulliksen result above fails for almost Artinian module. Let $R$
be a local ring such that $\dim R>0$. Take $T$ such that $T_i=R$,
$x: T_{i+1} \to T_i$ be the identity map for all $i$ and $S=R$.
Then $\ker(x) = T_0$ is clearly almost Artinian. However, $T$ is
not an almost Artinian $S[x]$-module, since this would immply, by
definition and \ref{kirby}, that each $T_i$ is Artinian for $i\gg
0$.
\end{eg}

\begin{lem}\label{length1}
Let  $I$ be an $R$-ideal that kills $M$. Assume $I$ is $m$-primary.
Then:
$$ \mu(M)\ell(R/I) \geq \ell(M) \geq \mu(M) $$
\end{lem}

\begin{proof}
The left inequality follows by tensoring the surjection:
$$ R^{\mu(M)} \to M \to 0$$
with $R/I$ to get :
$$ (R/I)^{\mu(M)} \to M \to 0$$
The right inequality follows from the surjection:
$$ M \to M/mM \to 0$$
\end{proof}

\begin{lem}\label{dual}
Let $ T= \bigoplus_{i\geq 0} T_i$ be a graded
$S=R[x_1,..,x_r]$-module, with all $x_i$ having equal negative
degrees. Suppose $f_R(T)=0$, (i.e that all $T_i$ have finite
length). Let $T^{\vee}$ denote $\bigoplus_{i}\Hom_R(T_i,E_{R}(k)) =
\Hom_{graded R-modules}(T, E_{R}(k))$, the graded Matlis dual. Then
$T^{\vee}$ becomes a graded module over  $S'=
R[x_1^{\vee},...,x_r^{\vee}]$. Furthermore, $T$ is an Artinian
$S$-module if and only if $T^{\vee}$ is a Noetherian $S'$-module.
\end{lem}

\begin{proof}
The first statement follows from the fact that Matlis dual is a
contravariant functor. For the second statement, we just note that
in this situation, being Noetherian (respectively Artinian) is the
same as satisfying the ACC (respectively DCC) condition on graded
submodules (one could use \ref{kirby} to see this). Since
$(-)^{\vee}$ gives a order-reserving bijection between the sets of
graded submodules of $T$ and $T^{\vee}$, we are done.
\end{proof}

\begin{lem}\label{P1}
Suppose that $T$ is almost Artinian over $S$, with $x_i$ of negative degree $-d_i$.
Then one can write :
$$ P_T(t)= \sum_i a\ell(T_i)t^i = {p(t)\over \prod_1^r (1-t^d_i)}$$
with $p(t)\in \ZZ[t]$.
\end{lem}

\begin{proof}

Let $N=f_R(T)$. Then the graded module $T'=T_{\geq N}$  is an
Artinian module over $S$. So we have:
$$ P_T(t)= P_{T'}(t)+ q(t) $$
for some $q(t) \in \ZZ[t]$. The result now follows from standard
facts on Artinian modules over graded rings.
\end{proof}

\section{Almost Artinian structure of
$\Tor_*^R(M,N)$}\label{almost_art}

Let $R,Q$ be local rings such that $R = Q/(f_1,..,f_r)$ and
$(f_1,...,f_r)$ is a regular sequence on $Q$. Let $M,N$ be finitely
generated $R$-modules. It is well known that $\Tor_*^R(M,N)$ has a
module structure over the ring of cohomology operators
$S=R[x_1,\dots,x_r]$. We begin by reviewing standard facts about
these  operators induced by $f_1,\dots,f_r$.

\begin{prop}\label{Eiops}
Let $R,Q,M,N$ be as above. Then for $1\leq j \leq r$ there are $R$-linear maps:
$$ x_j: \Tor_{i+2}^R(M,N) \to \Tor_i^R(M,N) $$
for $i \geq 0$ which satisfy the following properties:
\begin{enumerate}
\item $\Tor_*^R(M,N)$ becomes a graded module over the ring
$S=R[x_1,...,x_r]$ (with each $x_i$ has degree $-2$).

\item When $r=1$, the map $x_1$ is (up to sign) the connecting
homomorphism in the long exact sequence of $\Tor$s for change of
rings \ref{longexact}.

\item For any $1\leq s \leq r$, the operators $x_1,...,x_s$ act on
$\Tor_*^R(M,N)$ in two ways: the initial one, from $R =
Q/(f_1,...,f_r)$, and a new one, from the presentation $R
=R'/(f_1,...,f_s)$ with $R' = Q/(f_{s+1},...,f_r)$.

\end{enumerate}

\end{prop}

\begin{proof}
Part (1) is  \cite{Ei}, Proposition 1.6. Part (2) and (3) follows
from dualizing Proposition 2.3 in \cite{Av1}, which gave the
corresponding result for $\Ext$. Specifically, one should tensor the
exact sequence (2.5.1) there with $N$. One can also deduce (2) and
(3) from the detailed comparison of different constructions of
cohomology operators done in \cite{AS}, Section 4.

\end{proof}

In the following lemma we shall write $T^R(M,N)$ or sometimes $T^R$
for $\Tor_*^R(M,N)$.
\begin{lem}\label{AA}
Let $R,Q,S,M,N$ as in Proposition \ref{Eiops}. Assume that
$f_R(T^R(M,N))<\infty $. Then $T^R(M,N)$ is almost Artinian over $S$
if and only if $ T^Q(M,N)$ is almost Artinian over $Q$ (which is
equivalent to $\Tor_i^Q(M,N)=0$ for $i \gg 0$).
\end{lem}

\begin{proof}
Assume $T^R$ is almost Artinian over $S$. Let $R'= Q/(f_1,...,f_{r-1})$
and consider the long exact sequence :

\[ \xymatrix {\ar[r] & T_{i+2}^R  \ar[r]^{d_{i+2}} &  T_i^R  \ar[r] & T_{i+1}^{R'} \ar[r] & T_{i+1}^R  \ar[r]^{d_{i+1}} & } \]

Note that the connecting maps $d_i$s are just multiplication by
$x_r$ (up to sign) in $T^R$. Break the long exact sequence into
short exact sequences and assemble them together to get an exact
sequence:
$$ 0 \to T^R/x_rT^R \to T^{R'} \to \ker(x_r) \to 0 $$
The result follows from lemma \ref{3.9} and induction on $r$ (the
case $r=0$ is vacuous).

Now we prove the ``if" part. Here a more careful argument is needed.
This is because lemma \ref{gulik} fails badly for almost Artinian
modules (see \ref{exalmostart}). We still use induction. By
assumption there is an integer $n$ such that $T_{\geq n}$ is
Artinian. Truncate the long exact sequence for $\Tor$ above at
$T_{n+1}^R$. It is obvious that $\ell(T_i^{R'})<\infty$ for $i\geq
n+1$. By the induction hypotheses, $T^{R'}$ is almost Artinian,
hence so is $T^{R'}_{i\geq {n+1}}$. Consider the module $K =
\oplus_{i\geq {n+2}} \ker(d_i)$. Then $K$ is a graded quotient of
$T^{R'}_{i\geq {n+1}}$, therefore it is Artinian over $S$.
Restricting to $T^R_{i\geq{n}}$, $\ker(x_r) = K\oplus T^R_{n+1}
\oplus T^R_n$. Since the last two summmands are Artinian $R$-modules
and $K$ is an Artinian $S$- module, $\ker(x_r)$ is Artinian
$S$-module. It follows from lemma \ref{gulik} that $T^R_{i\geq{n}}$
is Artinian $S$-module, finishing our proof.

\end{proof}

\begin{cor}
If $\ell(M\tensor_R N)<\infty$, then $T^R(M,N)$ is Artinian over $S$
if and only if  $T^Q(M,N)$ is Artinian over $Q$ (if and only if
$\Tor_i^Q(M,N) = 0$ for $i \gg 0$).
\end{cor}

\section{Asymptotic behavior of lengths of $\Tor^R(M,N)$ and the function $\h^R(M,N)$
}\label{eta}

In this section we shall study the behavior of the $\lcx$ function
(see Section \ref{prelim}) over complete intersections.  We will
only be concerned with $R$-modules $M,N$ such that $f_R(M,N)=
f_R(\Tor_*^R(M,N))<\infty$, that is, $\ell(\Tor_i^R(M,N))<\infty$
for $i\gg 0$. For such pairs of modules, we define the
\textit{generalized Betti numbers} as :
$$  \beta_i^R(M,N) = a\ell(\Tor_i^R(M,N)) $$
Obviously, when $N=k$ we have the usual Betti numbers for $M$. The
following result shows that well-known properties for Betti numbers
still hold(cf. \cite{Av1}, 9.2.1):
\begin{thm}\label{beta1}
Let $R,Q$ be local rings such that $R = Q/(f_1,..,f_r)$ and
$f_1,...,f_r$ is a regular sequence on $Q$. Let $M,N$ be $R$ modules
such that $\Tor_i^Q(M,N)=0$ for $i\gg0$ (which is automatic if $Q$
is regular) and $\ell(\Tor_i^R(M,N))<\infty$ for $i\gg 0$. Let:
$$ P_{M,N}^R(t) = \sum_{i=0}^{\infty}\beta_i^R(M,N)t^i$$
We have:
\begin{enumerate}
\item
There is a polynomial $p(t) \in \mathbb{Z}[t]$ with $p(\pm 1)\neq 0$, such that
$$ P_{M,N}^R(t)  = \frac{p(t)}{(1-t)^c(1+t)^d}$$
\item
For $i\gg 0$, there are equalities:
$$ \beta_i(M,N) = \frac{m_0}{(c-1)!} i^{c-1} + (-1)^i\frac{n_0}{(d-1)!} i^{d-1}+ g_{(-1)^i}(i) $$
with $m_0 \geq 0$ and polynomials $g_{\pm}(t)\in \mathbb{Q}[t]$ of degrees $< \max\{c,d\}-1$.
\item
$d \leq c = \lcx_R(M,N) \leq r$.
\end{enumerate}

\end{thm}

\begin{proof}
\begin{enumerate}
\item
This is due to the fact that $\Tor_*^R(M,N)$ is an almost Artinian graded $R[x_1,...,x_r]$-module,
with the $x_i$s having degree $-2$, and \ref{P1}.
\item From part (1), we can write :
$$\sum_i \beta_it^i = \frac{p(t)}{(1-t)^c(1+t)^d} =\sum_{l=0}^{c-1}{\frac{m_l}{(1-t)^{c-l}}} +
                                         \sum_{l=0}^{d-1}{\frac{n_l}{(1+t)^{d-l}}} + q(t)  $$
Here $q(t) \in \mathbb{Z}[t]$.
Then by comparing coefficients from both sides we get the desired formula for $\beta_i$.
Since $\beta_i \geq 0$, $m_0$ must be greater than or equal to $0$.
\item
That $c \leq r$ is obvious. Since the sign of $\beta_i$ for odd $i$
is positive only if $c \geq d$, the first inequality is also clear.
The size of $\beta_i$ behaves like a polynomial of degree
$\max\{c,d\}-1= c-1$, which shows that $\lcx_R(M,N) = c$.
\end{enumerate}
\end{proof}

The behavior of $\beta_i(M,N)$ allows us to define an asymptotic
function on $M,N$ as follows:
\begin{dfn}
For a pair of $R$-modules $M,N$ such that $f_R(M,N)<\infty$, and an integer $e \geq \lcx_R(M,N)$,
let :
$$\h_e^R(M,N):= \lim_{n\to\infty} \frac{\sum_0^{n} (-1)^i \beta_i(M,N)}{n^e}$$
\end{dfn}

The following result shows that out $\h$ function behaves quite well:
\begin{thm}\label{beta2}
Let $R,Q$ be local rings such that $R = Q/(f_1,..,f_r)$ and
$f_1,...,f_r$ is a regular sequence on $Q$. Let $M,N$ be $R$ modules
such that $\Tor_i^Q(M,N)=0$ for $i\gg0$ and
$\ell(\Tor_i^R(M,N))<\infty$ for $i\gg 0$. Let $c = \lcx_R(M,N)$ and
$e \geq c$ be an integer.
\begin{enumerate}
\item  $\h_e^R(M,N)$ is finite, and if $e > c$, $\h_e^R(M,N) = 0$
\item (Biadditivity)  Suppose that $e \geq 1$. Let $0 \to M_1 \to M_2 \to M_3 \to 0$ and
$0 \to N_1 \to N_2 \to N_3 \to 0$ be exact sequences and suppose that for each $i$,
$f_R(M_i,N)$ and $f_R(M,N_i)$ are finite. If $e \geq \max_i\{\lcx(M_i,N)\}$ then:
$$ \h_e^R(M_2,N) = \h_e^R(M_1,N) + \h_e^R(M_3,N)$$
and if $e \geq \max_i\{\lcx(M,N_i)\}$, then:
$$ \h_e^R(M,N_2) = \h_e^R(M,N_1) + \h_e^R(M,N_3)$$
If $\ell(M\tensor_RN)<\infty$, the conclusion also holds for $e=0$.
\item (Change of rings) Let $r\geq 1$ and $R'=Q/(f_1,...,f_{r-1})$.
Note that we also have $\ell(\Tor_i^R(M,N))<\infty$ for $i\gg 0$.
Assume that  $e\ge 2$ and $e-1 \geq \lcx_{R'}(M,N)$. Then we have :
$$ \h_e^{R}(M,N) = \frac{1}{2e} \h_{e-1}^{R'}(M,N)$$

If $\ell(M\tensor_RN)<\infty$, $r=1$ and $e=1$ we have:
$$ \h_1^{R}(M,N) = \frac{1}{2}\h_{0}^{R'}(M,N)$$

\end{enumerate}
\end{thm}

\begin{proof}

Let $n>h$ be integers.  Let $g_{M,N}^R(h,n) = \sum_{i=h}^{n} (-1)^i \beta_i^R(M,N)) $. Then, for
a fixed $h$, it is easy to see that:
$$ \h_{e}^R(M,N) = \lim_{n\to\infty} \frac{g^R_{M,N}(h,n)}{n^e}$$

(1)
If $e=0$ then $\lcx_R(M,N)=0$, so $\beta_i =0$ for $i\gg0$, and there is nothing to prove.
Assume $e>0$. We choose an integer $h$ is big enough so that the formula for $\beta_i(M,N)$ in
\ref{beta1} is true for all $i \ge h$. Hence :

\begin{eqnarray}
g^R_{M,N}(h,n) & = &\sum_{i=h}^{n} (-1)^i \beta_i \nonumber \\
       & = &\frac{m_0}{(c-1)!}\sum_{h}^{n}(-1)^i i^{c-1}
                     + \frac{n_0}{(d-1)!}\sum_{h}^{n}i^{d-1}
                     +\sum_{h}^{n}(-1)^i g_{(-1)^i}(i) \nonumber
\end{eqnarray}
Since the first and third terms are of order $n^{c-1}$ or less, and $e\geq d$, we have:
\begin{eqnarray}
\lim_{n\to\infty} \frac{g^R_{M,N}(h,n)}{n^r} & =
& \lim_{n\to\infty} \frac{n_0}{(d-1)!}\frac{\sum_{h}^{n}i^{d-1}}{n^e} \nonumber\\
& = & \lim_{n\to\infty} \frac{n_0}{d!}n^{d-e}\nonumber
\end{eqnarray}
Since $d \leq c \leq e$, the limit is finite. The second statement
is also obvious.

(2)
It is enough to prove the first equation, since the second one follows in an identical
manner. The short exact sequence $0 \to M_1 \to M_2 \to M_3 \to 0$ gives rise
to the long exact sequence :
$$ ... \to \Tor_i(M_1,N) \to \Tor_i(M_2,N) \to \Tor_i(M_3,N) \to \Tor_{i-1}(M_1,N) \to...$$
which we truncate at $n > h > \max\{f_R(M_1,N),f_R(M_2,N),f_R(M_3,N)\}$ to get :

$ 0 \to B_n \to \Tor_n(M_1,N) \to \Tor_n(M_2,N) \to \Tor_n(M_3,N) \to ...\to \Tor_h(M_3,N) \\
\to C_h \to  0$
Taking the alternate sum of length we get:
$$ g_{M_1,N}(h,n) - g_{M_2,N}(h,n) + g_{M_3,N}(h,n) = \pm \ell(B_n) \pm \ell(C_h) $$
Divide by $n^e$ and take the limit as $n\to\infty$. Since $h$ is
fixed and $\ell(B_n)\leq \ell(\Tor_n(M_1,N)$, which is of order
$n^{\lcx(M_1,N)-1}$, the right hand side must be $0$. The left hand
side gives exactly the equality we seek.

(3) We will make use of the long exact sequence for change of
rings in \ref{longexact}:

$$ ... \to \Tor^R_{i-1}(M,N) \to \Tor^{R'}_i(M,N) \to \Tor_i^R(M,N) \to \Tor^R_{i-2}(M,N) \to...$$
Again, choose  $n> h> f_R(M,N)$ and we truncate as follows :
$$        \begin{array}{ll}

                          0 \to B_n  \to \Tor^R_{n+1}(M,N) \\
\to \Tor_{n-1}^R(M,N) \to \Tor_{n}^{R'}(M,N) \to \Tor_{n}^R(M,N)\\
\to  ...                                                          \\
\to \Tor_{h}^R(M,N) \to \Tor_{h+1}^{R'}(M,N) \to \Tor_{h+1}^R(M,N) \to C_h \to 0
\end{array}
$$\\
Take the alternate sum of lengths to get:
$$ g_{M,N}^{R'}(0,n) = (-1)^{n+1}\ell(\Tor^R_{n+1}(M,N)) + (-1)^n \ell(\Tor^R_{n}(M,N)) \pm \ell(B_n) \pm \ell(C_h) $$
Now we make use of Theorem \ref{beta1}, observing that since $B_n$
is a quotient of $\Tor^{R'}_{n+1}(M,N)$, $\ell(B_n)$ is of order
$n^{e-2}$ or less. Thus:
$$ g_{M,N}^{R'}(h,n) = \frac{n_0}{(d-1)!}(n^{d-1}+(n+1)^{d-1}) + f(n)$$
where $f(n)$ is of order $n^{e-2}$. Divide by $n^{e-1}$ and take
limit as $n\to\infty$, we get :
\begin{eqnarray}
\h^{R'}_{e-1}(M,N) & = &2d \lim_{n\to\infty} \frac{n_0}{d!}n^{d-e}\nonumber \\
                        & = &2e \lim_{n\to\infty} \frac{n_0}{d!}n^{d-e}\nonumber \\
                        & = &2e \h_{e}^R(M,N)\nonumber
\end{eqnarray}

The second equality follows because the limit is nonzero if and only
if $d=e$. The third equality follows from the last line of part (1).
The last statement of (3) can be proved in a similar (but simpler)
manner, and in any case can be found in \cite{Ho1}.
\end{proof}

\begin{rmk}\label{SHG}
We want to note that $\h_e^R(M,N)$ can be viewed as a natural
extension of some familiar notions. When $R$ is regular and
$\ell(M\tensor_RN)<\infty$, $\h_0^R(M,N)=\chi^R(M,N)$, Serre's
intersection multiplicity. When $R=Q/(f)$ is a hypersurface, then
$2\h_1^R(M,N)= \theta^R(M,N)$, the function defined by Hochster in
\cite{Ho1} and studied in details in \cite{Da1} and \cite{Da2}. When
$R$ is a complete intersection of codimension $r$ and
$\ell(M\tensor_RN)<\infty$ then $\h_r^R(M,N)$ agrees (up to a
constant factor) with a notion defined by Gulliksen in \cite{Gu}. We
recall Gulliksen's definition. When $R$ is a complete intersection
of codimension $r$ and $\ell(M\tensor_RN)<\infty$ then
$\Tor_*^R(M,N)$ is Artinian over $R[x_1,...,x_r]$ and we can write:
$$ P_{M,N}^R(t)  = \frac{p(t)}{(1-t^2)^r}$$
Then Gulliksen defined:
$$\chi^R(M,N) = p(-1) $$
Since  almost Artinian modules behave less well than Artinian ones
(see proof of \ref{AA}), and in any case finer information are
needed for our main results, we could not repeat Gulliksen's idea.
\end{rmk}

\section{Comparison of complexities}\label{comparison}

This section is devoted to the study of relationships between the
different notions of complexities over complete intersections (which
we defined in section \ref{prelim}). We again let $(R,\mathrm{m},k)$
and $(Q,\mathrm{n},k)$ be local rings such that $R =
Q/(f_1,..,f_r)$, $Q$ is regular and $(f_1,...,f_r)$ is a regular
sequence on $Q$. We denote by $x_i$, $i=1,...,r$ the cohomology
operators and let $S = R[x_1,...,x_r]$. Let $M,N$ be $R$-modules.

Throughout this section, let $\mathcal T=\Tor_*^R(M,N)$. The main
result states that if $f_R(M,N)=f_R(\mathcal T)<\infty$ then:
$$ \lcx_R(M,N) = \tcx_R(M,N) = \cx_R(M,N) $$

We first note an easy:

\begin{lem}\label{lem5.1}
Suppose there is an $m_R$-primary ideal $I$ such that $I$ kills
$\mathcal T_{\geq n}$ for some integer $n$. Then $\lcx_R(M,N) = \tcx_R(M,N)$.
\end{lem}

\begin{proof}
Suppose $i$ is an integer such that $I$ kills $\Tor_i^R(M,N)$. Let
$\mu_i = \ell(\Tor_i^R(M,N))$ and $a = \ell(R/I)$. Then by
\ref{length1} : $a\mu_i \geq \beta_i \geq \mu_i$. Since this is true
for all $i\gg 0$, we have $\cx(\{\mu_i\}) = \cx(\{\beta_i\})$, which
is what we want.
\end{proof}

Next, we recall some notations and results from Avramov and
Buchweitz's paper \cite{AB1}. Let $\widetilde{k}$ denotes an
algebraic closure of $k$. A residual algebraic closure of $Q$ is a
flat extension of local ring $Q \subseteq \widetilde Q$ such that
$\mathrm{n}\widetilde Q$ is the maximal ideal
$\widetilde{\mathrm{n}}$ of $\widetilde Q$ and the induced map $k
\to \widetilde Q/\widetilde {\mathrm{n}}$ is the embedding $ k
\hookrightarrow \widetilde k$. Such extension always exists (see
\cite{AB1}, 2.1).

Let $\mathcal E = \Ext_R^*(M,N)\tensor_R k$, a module over
$A=S\tensor_R k = k[x_1,...x_r]$. Then the  support variety
$V^*(Q,R,M,N)$ was defined as  the zero set (in $\widetilde k^r$)
of the annihilator of $\mathcal E$ in $A$, plus ${0}$. The key
result here shows why this is an important object:

\begin{thm}(\cite{AB1},2.4,2.5]\label{AB} Let $Q,R,M,N$ as above. Then :\\
(1) $\cx^R(M,N) = \dim V^*(Q,R,M,N)$\\
(2) Let $0 \neq \overline a=(\overline a_1,...,\overline a_r)\in
\widetilde{k}^r$.  Then $\overline a\in V^*(Q,R,M,N)$ if and only
if $\Ext_{Q_a}^n(\widetilde M, \widetilde N)\neq 0$ for infinitely
many $n$. Here $\widetilde M = M\tensor_R \widetilde Q$,
$a=(a_1,...,a_r)$ is a lift of $\overline a$ in $\widetilde{Q}^r$,
$f_a = \sum a_if_i$ and $Q_a = \widetilde{Q}/(f_a)$.

\end{thm}

Now, suppose $n=f_R(\mathcal T))<\infty$. Then $\mathcal T$ is
almost Artinian, hence  $\mathcal T_{\geq n}$ is an Artinian module
over $S$. Let $T_{\geq n}^{\vee}$ denote $\oplus_{i\geq
n}\Hom_R(\mathcal T_i, E_{R}(k)) = \Hom_{graded R-modules}(\mathcal
T, E_{R}(k))$ . Then $\mathcal T_{\geq n}^{\vee}$ becomes a graded
module over  $S'= R[x_1^{\vee},...,x_r^{\vee}]$ . Let $\mathcal
D=\mathcal T_{\geq n}^{\vee}\tensor_R k$, a module over $A'=
k[x_1^{\vee},...,x_r^{\vee}]$. Since $\mathcal T_{\geq n}$ is
Artinian over $S$, $\mathcal D$ is Noetherian over $A'$ by
\ref{dual}.

From the discussion above  we can define a set :
$$ V_*(Q,R,M,N) = \Z(\ann_{A'}\mathcal D) \cup \{0\} \subseteq \widetilde{k}^r.$$
The next result is a dual version of theorem \ref{AB}:

\begin{lem}\label{AB1}
Let $Q,R,M,N$ be as above such that $f_R(\mathcal T)<\infty$ (i.e.
$\ell(\Tor_i^R(M,N))<\infty$ for $i\gg0$). Then :
\\ (1) $\tcx^R(M,N) = \dim V_*(Q,R,M,N)$.
\\ (2) Let $0 \neq \overline a=(\overline a_1,...,\overline a_r)\in
\widetilde{k}^r$.  Then $\overline a\in V_*(Q,R,M,N)$ if and only
if $\Tor^{Q_a}_n(\widetilde M, \widetilde N)\neq 0$ for infinitely
many $n$. Here $\widetilde M = M\tensor_R \widetilde Q$,
$a=(a_1,...,a_r)$ is a lift of $\overline a$ in $\widetilde{Q}^r$,
$f_a = \sum a_if_i$ and $Q_a = \widetilde{Q}/(f_a)$.

\end{lem}

\begin{proof}
(1) Let $n = f_R(\mathcal T)$, then $\mathcal T_{\geq n}$ is
Artinian. Since $\mathcal T_{\geq n}^{\vee}$ is Noetherian, by
\ref{cor1} the annihilators of $\mathcal T_i^{\vee}$ become
periodic of period 2 eventually. So the same thing is true for the
annihilators of $\mathcal T_i$. Lemma \ref{lem5.1} shows that $
\lcx_R(M,N) = \tcx_R(M,N)$. We now have:
$$\tcx_R(M,N) = \lcx_R(M,N) = \cx_R(\mathcal T)= \cx_R(\mathcal T_{\geq n}^{\vee})
= \cx_R(\mathcal D) =\dim(V_*^R(M,N))$$

(2) Without loss of generality, we may assume $\overline a=
(1,0,..,0)$ and $Q=\widetilde Q$. Then $f_a=f_1$. Let $Q_1 =
Q/(f_1)$. Then $R = Q_1/(f_2,..,f_r)$ and so $\mathcal T$ is a
module over $S_1 = R[x_2,...,x_r]$ (the actions here agree  with the
actions from $S$). By Proposition \ref{dual}, $\Tor_*^{Q_1}(M,N)$ is
almost Artinian over $Q_1$ $\iff$ $\mathcal T$ is almost Artinian
over $S$
 $\iff$  $\mathcal D$ is a finite module over
$A'_1 = k[x_2^{\vee},...,x_r^{\vee}]$ $\iff$ $(1,0,...,0) \notin
V_*(Q,R,M,N)$
\end{proof}

In summary:

\begin{thm}\label{compare}
Let $R$ be a local complete intersection and $M,N$ be $R$ modules
such that $f_R(\mathcal T)<\infty$. Then $\lcx_R(M,N) =
\tcx_R(M,N) = \cx_R(M,N)$.
\end{thm}

\begin{proof}
The first equality was proved in part (1) of Lemma \ref{AB1}. The
second equality follows part (2) of  \ref{AB}, part (2) of \ref{AB1}
and the fact that over the hypersurface $Q_a$,
$\Tor^{Q_a}_n(\widetilde M, \widetilde N)\neq 0$ for infinitely many
$n$ if and only if $\Ext_{Q_a}^n(\widetilde M, \widetilde N)\neq 0$
for infinitely many $n$ (see \cite{AB1}, 5.12 and \cite{HW1}, 1.9).

\end{proof}

As a corollary we reprove a result by Avramov and Buchweitz, which
says that the vanishing of all higher $\Tor$ modules is equivalent
to the vanishing of all higher $\Ext$ modules.

\begin{cor}(\cite{AB1}, 6.1)\label{AB2}
 Let $R$ be a local complete intersection and $M,N$ be $R$-modules. Then $\Ext_R^i(M,N) = 0$ for $i \gg0$ if and only if
$\Tor_i^R(M,N) = 0$ for $i \gg0$.
\end{cor}

\begin{proof}
The statement is equivalent to saying that $\cx_R(M,N)=0$ if and
only if $\tcx_R(M,N)=0$. We will use induction on $d = \dim R$. If
$\dim R =0$ then all the  $\Tor$ modules have finite length, so we
can apply \ref{compare} directly. If $d>0$, and $\Tor_i^R(M,N) = 0$
for $i \gg0$, then \ref{compare} also applies. Assume that
$\Ext_R^i(M,N) = 0$ for $i \gg0$.  By localizing at prime ideals in
the punctured spectrum of $R$ and using induction hypothesis, we get
that all the high $\Tor$ modules have finite length, so we can apply
\ref{compare} again.
\end{proof}

\begin{rmk}
It is not known whether $\cx_R(M,N) = \tcx_R(M,N)$ when $R$ is a
local compete intersection and $(M,N)$ is any pair of $R$-modules.
Avramov and Buchweitz's result says that $\cx_R(M,N)=0$ if and only
if $\tcx_R(M,N)=0$. Our result \ref{compare} shows that $\cx_R(M,N)
= \tcx_R(M,N)$ when $f_R(M,N)<\infty$.
\end{rmk}

\begin{cor}\label{AB3}
Let $R$ be a a local complete intersection and $M,N$ be $R$-modules
such that $f_R(M,N)<\infty$. Then $\tcx_R(M,N) \leq
\min\{\cx_RM,\cx_RN\}$.
\end{cor}

\begin{proof}
This follows from theorem \ref{compare} and the fact that
$\cx_R(M,N)\leq \min\{\cx_RM,\cx_RN\}$ (corollary 5.7, \cite{AB1}).
\end{proof}

\section{c-rigidity}\label{c-rigid}

We will say that $(M,N)$ is \textit{c-rigid} if the vanishing of
$c$ consecutive $\Tor$'s forces the vanishing of all higher
$\Tor$'s. The module $M$ is called \textit{c-rigid} if $(M,N)$ is
\textit{c-rigid} for all finitely generated $N$. When $c=1$ we
simply say that $(M,N)$ (or $M$) is \textit{rigid}. We first
recall some notations and results from \cite{Ho1} and \cite{Da1}:

A local ring $(R,m,k)$ is a \textit{admissible complete
intersection} if  $\hat R \cong Q/(f_1,...,f_r)$, $f_1,...,f_r$ form
a regular sequence on $Q$ and  $Q$ is a power series ring over a
field or a discrete valuation ring.\\

\textbf{The function $\theta^R(M,N)$}\\

Let $R = T/(f)$ be an admissible local hypersurface. The function
$\theta^R(M,N)$ was introduced by Hochster ([Ho1]) for any pair of
finitely generated modules $M,N$ such that $f_R(M,N)<\infty$  as:
$$ \theta^R(M,N) = \ell(\Tor_{2e+2}^R(M,N)) - \ell(\Tor_{2e+1}^R(M,N)) .$$
where $e$ is any integer $\geq (d+2)/2$. It is well known (see
\cite{Ei}) that $\Tor^R(M,N)$ is periodic of period 2 after $d+1$
spots, so this function is well-defined. The theta function
satisfies the following properties. First, if $M\tensor_RN$ has
finite length, then:
$$\theta^R(M,N) = \chi^T(M,N).$$
Secondly, $\theta^R(M,N)$ is biadditive on short exact sequence,
assuming it is defined. Specifically, for any short exact
sequence:
$$0 \to N_1 \to N_2 \to N_3 \to 0$$
and any module $M$ such that $f_R(M,N_i)<\infty$ for all
$i=1,2,3$, we have $\theta^R(M,N_2) = \theta^R(M,N_1) +
\theta^R(M,N_3)$. Similarly, $\theta(M,N)$ is additive on the
first variable.

In \cite{Da1}, we show that when $\theta^R(M,N)$ can be defined
and vanishes, then $(M,N)$ is rigid:

\begin{prop}\label{rg1}
Let $R$ be an admissible hypersurface and $M,N$ be $R$-modules
such that $f_R(M,N)<\infty$ (so that $\theta^R(M,N)$ can be
defined). Assume $\theta^R(M,N)=0$. Then $(M,N)$ is rigid.
\end{prop}

By Remark \ref{SHG}, $\theta^R(M,N)=2\h^R_1(M,N)$ (when both are
defined). Using our generalized function $\h^R(M,N)$, it is easy to
get similar results on c-rigidity over complete intersections. We
first isolate a simple corollary of \ref{longexact}, whose proof we
will omit:

\begin{cor}\label{cor6.2}
Let $Q$ be a Noetherian ring with $f$ a nonzerodivisor on $Q$. Let
$R =Q/(f)$ and $M,N$ be $R$-modules. Let $n,i,c$ be integers. \\
(1) If $\Tor_Q^i(M,N)=0$ for all $i\geq n$ then $\Tor_i^R(M,N)
\cong \Tor_{i+2}^R(M,N)$ for all $i\geq n-1$.\\
(2) If $\Tor_R^i(M,N)=0$ for $n \leq i \leq n+c$ then
$\Tor_Q^i(M,N)=0$ for $n+1 \geq i \geq n+c$.
\end{cor}

\begin{thm}\label{rig2}
Let $R$ be a codimension $r>0$ admissible complete intersection
and $M,N$ be $R$-modules. Assume that $f_R(M,N)<\infty$ and
$\h_r^R(M,N) = 0$. Then $(M,N)$ is $r$-rigid.
\end{thm}

\begin{proof}
We use induction on $r$. The case $r=1$ is Proposition \ref{rg1}.
We may assume $R=Q/(f_1,...,f_r)$ where $Q$ is regular. Suppose
$\Tor_R^i(M,N)=0$ for $n+1 \geq i \geq n+r$ for some integer $n$.
Let $R'=Q/(f_2,...,f_r)$. Then by Corollary \ref{cor6.2} we have
$\Tor_{R'}^i(M,N)=0$ for $n+2 \leq i \leq n+r$. By Theorem
\ref{beta2} $\h_{r-1}^{R'}(M,N)=0$ so by induction hypothesis
$(M,N)$ is $(r-1)$-rigid as $R'$ module. Thus $\Tor_{R'}^i(M,N)=0$
for $i \geq n+2$ and by \ref{cor6.2} again we have
$\Tor_R^i(M,N)=0$ for $i \geq n+1$.
\end{proof}

\begin{cor}\label{rig2.1}
Let $R$ be a codimension $r$ admissible complete intersection and
$M,N$ be $R$-modules. Assume the $\pd_{R_p}M_p <\infty$ for all
$p\in Y(R)$ and $[N]=0$ in $\overline{G}(R)_{\mathbb{Q}}$, the
reduced Grothendieck group finitely generated modules over $R$ with
rational coefficients. Then $(M,N)$ is $r$-rigid.
\end{cor}

\begin{proof}
The proof is identical to the hypersurface case (see \cite{Da1},
4.3) using \ref{rig2} instead of \ref{rg1}.
\end{proof}

\begin{cor}\label{rig2.2}
Let $R$ be an codimension $r>0$ admissible complete intersection
and $M$ be an $R$-module such that $[M] =0$ in $\overline
G(R)_{\mathbb{Q}}$. Let $IPD(M):= \{ p\in \Spec(R) |\ \pd_{R_p}M_p
=\infty \}$. Assume that $IPD(M)$ is either $\emptyset$ or is
equal to $\Sing(R)$. Then $M$ is $r$-rigid.
\end{cor}

\begin{proof}
The proof is identical to the hypersurface case (see \cite{Da1},
4.5) using \ref{rig2} instead of \ref{rg1}.
\end{proof}

\begin{cor}\label{rig2.3}
Let $R$ be a codimension $r>0$ admissible complete intersection
and $M,N$ be $R$-modules. Assume:\\
(1) $M\tensor_RN$ has finite length.\\
(2) $\dim(M) + \dim(N) \leq \dim(R)+r-1$.\\
Then $(M,N)$ is $r$-rigid. In particular, a finite length module
is $r$-rigid
\end{cor}

\begin{proof}
The case $r=1$ is 4.4 in \cite{Da1}. Then \ref{cor6.2} allows us
to use induction.
\end{proof}

\begin{rmk}
We want to point out that instead of using $\h_c^R(M,N)$, one can
appeal to $\theta^{R_1}(M,N)$ (here $R_1 = Q/(f_1)$) to prove some
of the above results. We will sketch a proof for \ref{rig2.1}. The
point is that the hypotheses on $M$ and $N$ lift to $R_1$. So $M,N$
are rigid over $R_1$. Now using the change of rings exact sequence
repeatedly shows that $M,N$ are c-rigid over $R$.
\end{rmk}

We also note this generalization of a result by Lichtenbaum in
{\cite{Li}) which says that over an admissible hypersurface, a
module of finite projective dimension is rigid:
\begin{cor}
Let $R$ be a codimension $r>0$ admissible complete intersection,
and $M$ be an $R$-module such that $\cx(M)\leq r-1$. Then  $M$ is
$r$-rigid.
\end{cor}

\begin{proof}
Let $N$ be any $R$-module we want to show that $(M,N)$ is
$r$-rigid. We use induction on $\dim R$. If $\dim R=0$, all
modules have finite length, so by \ref{AB3} we have $\tcx(M,N)
\leq \cx(M)$. Thus $\h_r^R(M,N)=0$ and $(M,N)$ is $r$-rigid by
Theorem \ref{rig2}. Assume $\dim R>0$. By induction (through
localizing at all primes $p \in Y(R)$) we have $f_R(M,N)<\infty$
so we can apply \ref{AB3} and \ref{rig2} again to show that
$\h_r^R(M,N)=0$.
\end{proof}

In general, when a powerful rigidity result is not present, we have
to be content with ``being rigid after certain point". We will say
that $(M,N)$ is \textit{(c,n)-rigid} if  $\Tor_i^R(M,N)=0$ for
$N\leq i\leq N+c-1$ with $N>n$ forces the vanishing of all
$\Tor_i^R(M,N)$ for $i \geq N$. The module $M$ is called
\textit{(c,n)-rigid} if $(M,N)$ is \textit{(c,n)-rigid} for all
finitely generated $N$.

\begin{lem}\label{rig3.1}
Let $Q$ be a local ring with $f$ a nonzerodivisor on $Q$. Let $R
=Q/(f)$ and $M,N$ be $R$-modules such that $f_R(M,N)<\infty$ and
$\pd_Q M<\infty$ (so $\h_1^R(M,N$ can be defined by \ref{beta2}). If
$\h_1^R(M,N)=0$ then $(M,N)$ is $(1,\depth R -\depth M)$-rigid.
\end{lem}

\begin{proof}
Since $\pd_Q M<\infty$ we have $\Tor_i^R(M,N)=0$ for $i> \depth Q-
\depth M$. By \ref{cor6.2} we have $\ell(\Tor_i^R(M,N)) =
\ell(\Tor_{i+2}^R(M,N))$ for all $i> \depth Q- \depth M-1 = \depth
R-\depth M$. But then the condition $\h_1^R(M,N)=0$ forces
$\ell(\Tor_i^R(M,N)) = \ell(\Tor_{i+1}^R(M,N))$ for all $i> \depth
R-\depth M$ and the conclusion follows trivially.
\end{proof}

\begin{thm}\label{}
Let $R,Q$ be local rings such that $R = Q/(f_1,..,f_r)$ and
$f_1,...,f_r$ is a regular sequence on $Q$. Let $M,N$ be $R$ modules
such that $f_R(M,N)<\infty$ and $\pd_Q M<\infty$ (so $\h_r^R(M,N$
can be defined by \ref{beta2}). If $\h_r^R(M,N)=0$ then $(M,N)$ is
$(r,\depth R -\depth M)$-rigid.
\end{thm}

\begin{proof}
The proof goes by induction using \ref{rig3.1} as the base case and
\ref{beta2} for the inductive step similarly to the proof of
\ref{rig2}.
\end{proof}

\section{Applications to tensor products over complete intersections}~\label{appsec}

In this section we aim to prove  results that can be viewed as
generalizations of Auslander's classical theorem that, over a
regular local ring $R$, $M\tensor_RN$ is torsion-free implies that
$\Tor_i^R(M,N)=0$ for $i\geq 1$. This conclusion forces $M,N$ to
satisfy the ``depth formula" :
$$ \depth(M) + \depth(N) = \depth(R) + \depth(M\tensor_RN).$$

An attempt to extend Auslander's result for complete intersections
was made in \cite{HW1} and \cite{HJW}. We first collect some
notation and results. Recall that $X^i(R)$ denotes the set of prime
ideals $p$ of height less than or equal to $i$ in $R$ (since we only
consider complete intersections, this also means $\depth(R_p)\leq
i$).

\textbf{The Condition $(\Se_n)$}

For a non-negative integer $n$, $M$ is said to satisfy $(\Se_n)$
if:
$$ \depth_{R_p}M_p \geq \min\{n,\dim(R_p)\} \ \forall p\in \Spec(R)$$
(The depth of the $0$ module is set to be $\infty$). This
definition was taken from ``Syzygies" (see [page3, EG]).

\textbf{The Pushforward}

Let $R$ be a Gorenstein ring and $M$ a torsion-free (equivalent to
$(\Se_1)$) $R$-module. Consider a short exact sequence : $$0 \to W
\to R^{\lambda} \to M^* \to 0 $$ Here $\lambda $ is the minimal
number of generators for $M^*$. Dualizing this short exact sequence
and noting that $M$ embeds into $M^{**}$ we get an exact sequence:
$$ 0 \to M \to R^{\lambda} \to M_1 \to 0$$
This exact sequence is called the \textit{pushforward} of $M$.

We record a result on pushforward in \cite{HJW} below for the
reader's convenience. Note that since their definition of $(\Se_n)$
contains some inconsistency with the literature, some minor details
need to be fixed. See \cite{HW3} for details, and also proof of
\ref{vanishingprop} below.

\begin{prop}\label{pushforward}(\cite{HJW}, 1.6)
Let $R,M,M_1$ as above. Then for any $p\in \Spec(R)$:\\
(1) $M_p$ is free if and only if $(M_1)_p$ is free. \\
(2) If $M_p$ is a maximal Cohen-Macaulay $R_p$-module, then so is $(M_1)_p$.\\
(3) $\depth_{R_p}(M_1)_p \geq \depth_{R_p}M_p -1$.\\
(4) If $M$ satisfies $(\Se_k)$, then $M_1$ satisfies
$(\Se_{k-1})$.
\end{prop}

\begin{prop}\label{vanishingprop}(\cite{HJW}, 2.1)
Let $R$ be a Gorenstein ring, let $c\geq1$ and let $M,N$ be $R$-modules such that:\\
(1) $M$ satisfies $(\Se_c)$.\\
(2) $N$ satisfies $(\Se_{c-1})$.\\
(3) $M\tensor_RN$ satisfies $(\Se_c)$.\\
(4) $M_p$ is free for each $p\in X^{c-1}(R)$.\\
Put $M_0:=M$ and, for $i = 1,...,c$, let
$$  0\to M_{i-1} \to F_i \to M_i \to 0 $$
be the pushforward. Then $\Tor_i^R(M_c,N) = 0$ for $i=1,...,c$.
\end{prop}

\begin{proof}
In view of \cite{HW3}, the proof in \cite{HJW} only needs one more
line. In the inductive step, instead of looking at any $p\in
\Spec(R)$, we localize at $p\in \Supp(M_i\tensor_RN)$ (which means
$p$ is also in $\Supp(M_i),\Supp(N)$).

\end{proof}

For the reader's conveniences we collect below what have been done
to extend Auslander's theorem mentioned at the beginning of this
section.
\begin{thm}
(\cite{HW1}, 2.7) Let $R$ be a hypersurface and $M,N$ be finitely
generated $R$-modules, one of which has constant rank (on the
associated primes of $R$). If $M\tensor_RN$ is reflexive, then
$\Tor_i^R(M,N)=0$ for $i\geq 1$ .
\end{thm}

\begin{thm}
(\cite{HJW}, 2.4). Let $R$ be a codimension 2 complete
intersection and $M,N$ be finitely generated
$R$-modules. Assume:\\
(1) $M$ is free of constant rank on $X^1(R)$.\\
(2) $N$ is free of constant rank on $X^0(R)$.\\
(3) $M,N$ satisfy $(\Se_2)$.\\
If $M\tensor_RN$ satisfies $(\Se_3)$, then $\Tor_i^R(M,N)=0$ for
$i\geq 1$.
\end{thm}

\begin{thm}
(\cite{HJW}, 2.8). Let $R$ be a codimension 3 admissible complete
intersection and $M,N$ be finitely generated
$R$-modules. Assume:\\
(1) $M$ is free of constant rank on $X^2(R)$.\\
(2) $N$ is free of constant rank $r$ on $X^1(R)$ and
$(\wedge^rN)^{**}\cong R$.(Such a module
$N$ is called orientable).\\
(3) $M,N$ satisfy $(\Se_3)$.\\
If $M\tensor_RN$ satisfies $(\Se_4)$, then $\Tor_i^R(M,N)=0$ for
$i\geq 1$.
\end{thm}

Our aim is to prove the following result:

\begin{thm}\label{vanishingcomplete}
Let $c$ be any integer greater or equal to $1$.  Let $R$ be a
codimension $c$ admissible complete intersection
and $M,N$ be finitely generated $R$-modules. Assume:\\
(1) $M$ is free on $X^c(R)$.\\
(2) $M,N$ satisfy $(\Se_c)$.\\
If $M\tensor_RN$ satisfies $(\Se_{c+1})$, then $\Tor_i^R(M,N)=0$
for $i\geq 1$.
\end{thm}

Our strategy will be to use the c-rigidity results in the previous
section. The next lemma is critical for our proof:

\begin{lem}\label{vanishinglemcomplete}
Let $R$ be an admissible complete intersection of codimension
$c>0$ and $M,N$ be $R$-modules.
Assume that:\\
(1) $\Tor_i^R(M,N)=0$ for $1\leq i \leq c$. \\
(2) $\depth(N)\geq 1$ and $\depth(M\tensor_RN)\geq 1$. \\
(3) $f_R(M,N)<\infty$ .\\
Then $\Tor_i^R(M,N)=0$ for $i\geq 1$.
\end{lem}

\begin{proof}
The depth assumptions ensure that we can choose $t$ a nonzero
divisor for both $N$ and $M\tensor_RN$. Let $\bar{N} = N/tN$.
Tensoring the short exact sequence :
 \[ \xymatrix {0 \ar[r] &N \ar[r]^{t} &N \ar[r] &\bar{N} \ar[r] &0} \]
with $M$ we get :
\[ \xymatrix{0 \ar[r] &\Tor_1^R(M,\bar{N}) \ar[r] &M\tensor_RN \ar[r]^{t} &M\tensor_RN \ar[r] &M\tensor_R \bar{N} \ar[r] &0}\]
which shows that $\Tor_1^R(M,\bar{N}) = 0$. Together with
condition (1), this shows $\Tor_i^R(M,\bar{N})=0$ for $1\leq i
\leq c$. But condition (3) is satisfied for both pair $M,N$ and
$M,\bar{N}$ so :
$$ \h_c^R(M,\bar{N}) = \h_c^R(M,N) - \h_c^R(M,N)= 0$$
The conclusion then follows from theorem \ref{rig2} and Nakayama's
lemma.
\end{proof}

\begin{proof} (of theorem \ref{vanishingcomplete}).
We use induction on $d = \dim (R)$. For $d \leq c$, condition (1)
forces $M$ to be free. Suppose $d = c+1$. Following proposition
\ref{vanishingprop} we have a sequence of modules $M_0 =M$,
$M_1,..M_c$ such that there are exact sequences:
$$ 0\to M_{i-1} \to F_i \to M_i \to 0$$
By the conclusion of \ref{vanishingprop}, for each $i$,
$\Tor_j^R(M_i,N) =0$ for $1\leq j \leq i$. So we have a short
exact sequence :
$$ 0 \to M_{i-1}\tensor_RN \to F_i\tensor_RN \to M_i\tensor_RN \to 0$$
By assumption (2), (3) and the fact that $d = c+1$ we have $\depth
(M\tensor_RN)\geq c+1$ and $\depth (N) \geq c$. Using the ``depth
lemma" repeatedly, we can conclude that $\depth(M_c\tensor_RN) \geq
1$. Now, condition (1) means that $M$ is free on the punctured
spectrum of $R$, and so $f_R(M_c,N) <\infty$. So lemma
\ref{vanishinglemcomplete} applies.

Now assume $d>c+1$. By the induction hypothesis, $\Tor_i^R(M,N)$
has finite length for $i\geq 1$. So an identical argument to the
one above together with lemma \ref{vanishinglemcomplete} give the
desired conclusion.
\end{proof}

\section{Some applications on intersection theory over complete
intersections}\label{intersection}

We present here a few other examples where one can exploit the
properties of the function $\eta^R$ proved in the last chapter. The
first one extends  Hochster's result on dimensional inequality
(\cite{Ho1}).

\begin{thm}\label{heightstable}
Let $R = Q/(f_1,..,f_r)$ be a local complete intersection with $Q$ a
regular local ring satisfying Serre's Positivity conjecture (if
$r=0$, then $R=Q$). Let $M$ be an $R$-module such that $\Supp(M)$
contains the singular locus $\Sing(R)$ of $R$ and $[M]=0$  in
$\overline G(R)_{\mathbb{Q}}$. Then for any $R$ module $R$ such that
$M\tensor_RN$ is finite length we have: $\dim M + \dim N < \dim R +
r$.
\end{thm}

\begin{proof}
Since $\Supp(M)\cap\Supp(N) = \{m\}$ we have $\Sing(R) \cap \Supp(N)
\subset \{m\}$. So $IPD(N) \subset \{m\}$. Therefore for any
$R$-module $N'$, $f_R(N',N)<\infty$ and $\h_r^R(N',N)$ is defined.
Then the fact that $[M]=0$  in $\overline G(R)_{\mathbb{Q}}$ and the
biadditivity of $\h$ (part (2), \ref{beta2}) force  $\h_r^R(M,N)=0$.
Then part (3) of \ref{beta2} indicates that $\chi^Q(M,N)=0$ and we
must have $\dim M +\dim N < \dim Q = \dim R+r$.

\end{proof}

The following example shows that the inequality above is sharp:

\begin{eg}
Let $k$ be an algebraically closed field of characteristic $0$. Let
$Q=k[x_{ij}], 0\le i\le r, 1\le j\le r$ and $m = (x_{ij})$. For
$1\le i\le r$, let $y_i=\sum_{j=1}^{r}x_{0j}x_{ij}$. Let
$R=Q_m/(y_1,,,y_r)$. Then $R$ is a local complete intersection of
codimension $r$ and $\dim R= r^2$.

Let $A$ be the square matrix $[x_{ij}]_{1\le i\le r, 1\le j\le r}$
and $d = \det(A)$. Let $I = (x_{01},...,x_{0r},d)$ and $J$ be the
ideal generated by the entries of $A$. Then
$$\dim R/I + \dim R/J = (r^2-1)+r= \dim R +r-1$$

The short exact sequence :
\[ \xymatrix {0 \ar[r] & R/(x_{01},...,x_{0r}) \ar[r]^{d} & R/(x_{01},...,x_{0r}) \ar[r] & R/I \ar[r] & 0} \]
shows that $R/I$ is $0$ in the Grothendieck group of $R$. It remains
to show $V(I)$ contains the singular locus of R. But $\Sing(R) =
V(I')$, where $I'$ is generated by the $r$-by-$r$ determinants of $B
= \left [ \frac{\partial(y_l)}{\partial x_{ij}} \right ] $. It is
easy to see that $I'$ contains $d,x_{01}^r,...,x_{0r}^r$, therefore
$V(I) \supseteq V(I')$ (they are actually equal).
\end{eg}

The next corollary gives an asymptotic version of the Vanishing
Theorem for complete intersections due to Roberts (cf.
\cite{Ro},13.1.1):

\begin{cor}\label{asym1}
Let $R = Q/(f_1,..,f_r)$ be a local complete intersection with $Q$ a
regular local ring and $M,N$ be finitely generated $R$-modules such
that $M\tensor_RN$ has finite length. Let $a=
\max\{\cx_R(M),\cx_R(N)\}$. Suppose that
 $\dim M + \dim N < \dim R + a$. Then $\eta_a^R(M,N)=0$.
\end{cor}

\begin{proof}
Let $b=r-a$. By theorem 9.3.1 in \cite{Av2} we can factor the
surjection $Q\to R$ as $Q\to R' \to R$ such that the kernels of both
maps are generated by regular sequence, with the first one having
length $b$, and $\cx_{R'}(M) = \cx_{R'}(N) = 0$. Applying Roberts'
theorem and \ref{beta2}, we see that the result follows.
\end{proof}

In our view, the greatest potential of this study is the link
between asymptotic homological algebra, a new and rapidly developing
field, and the classical homological questions. To further
illustrate this link, let us look at a well-known unsolved question
(see \cite{PS}), inspired by Serre's results on intersection
multiplicity:
\begin{ques}
Let $R$ be a local ring and $M,N$ be $R$-modules such that
$\ell(M\tensor_RN)<\infty$ and $\pd_RM <\infty$. Is it always true
that  $\dim M +\dim N \leq \dim R$?
\end{ques}

In view of \ref{asym1}, we would like to pose the following
``asymptotic" form of the above question:
\begin{ques}\label{asymp_serre}
Let $R$ be a local complete intersection and $M,N$ be $R$-modules
such that $\ell(M\tensor_RN)<\infty$. Is it always true that:
\begin{enumerate}
\item $\dim M +\dim N \leq \dim R + \cx(M)$?
\item $\dim M +\dim N \leq \dim R + \tcx(M,N)$?
\end{enumerate}

\end{ques}

Obviously, in view of \ref{AB3}, part (2) of the above (if true)
would be stronger than part (1). Also, part (1) and (2) are
equivalent for hypersurafces and hold true if $\hat R$ is a
hypersurface in a unramified or equicharacteristic regular local
ring (cf. \cite{HW1}, 1.9 and \cite{Da1}, 2.5).

\end{document}